\documentclass[ahp]{birkmult-ppn}
\usepackage{setspace}
\usepackage{amssymb}

\newtheorem{lemma}{Lemma}
\newtheorem{proposition}[lemma]{Proposition}
\newtheorem{corollary}[lemma]{Corollary}
\newtheorem{theorem}[lemma]{Theorem}

\newcommand{\be}[1]{\begin{equation}\label{#1}}
\newcommand{\ee}{\end{equation}}

\newcommand{\R}{\mathbb{R}}

\def\cprime{$'$}
\newcommand{\nnrm}[1]{|\kern-1pt|\kern-1pt|#1|\kern-1pt|\kern-1pt|}
\newcommand{\coupling}{\kappa}

\begin{document}
%%%%%%%%%%%%%%%%%%%%%%%%%%%%%%%%%%%%%%%%%%%%%%%%%%%%%%%%%%%%%%%%%%%%%%
%%%%%%%%%%%%%%%%%%%%%%%%%%%%%%%%%%%%%%%%%%%%%%%%%%%%%%%%%%%%%%%%%%%%%%
\title[Uniqueness results in supercritical equations]
{Non-existence and uniqueness results for supercritical semilinear elliptic equations}

\author[J. Dolbeault]{Jean Dolbeault}
\address{Ceremade (UMR CNRS no. 7534), Universit\'e Paris-Dauphine, Place de Lattre de Tassigny, 75775 Paris C\'edex~16, France.}
\email{dolbeaul@ceremade.dauphine.fr}
\author[R. Sta\'nczy]{Robert Sta\'nczy}
\address{Instytut Matematyczny, Uniwersytet Wroc\l awski, pl. Grunwaldzki 2/4, 50-384 Wroc\l aw, Poland.}
\email{stanczr@math.uni.wroc.pl}

% 35Axx General theory

% 35A05 General existence and uniqueness theorems
%
% 35Jxx Partial differential equations of elliptic type [See also 58J10, 58J20]
% 35J15 General theory of second-order, elliptic equations
% 35J20 Variational methods for second-order, elliptic equations
% 35J25 Boundary value problems for second-order, elliptic equations
% 35J60 Nonlinear PDE of elliptic type
% 35J65 Nonlinear boundary value problems for linear elliptic PDE; boundary value problems for nonlinear elliptic PDE
% 82C22 Interacting particle systems


\date{January 1st, 2009}

\begin{abstract}
Non-existence and uniqueness results are proved for several local and non-local supercritical bifurcation problems involving a semilinear elliptic equation depending on a parameter. The domain is star-shaped and such that a Poincar\'e inequality holds but no other symmetry assumption is required. Uniqueness holds when the bifurcation parameter is in a certain range. Our approach can be seen, in some cases, as an extension of non-existence results for non-trivial solutions. It is based on Rellich-Poho\-\v{z}aev type estimates. Semilinear elliptic equations naturally arise in many applications, for instance in astrophysics, hydrodynamics or thermodynamics. We simplify the proof of earlier results by K. Schmitt and R.~Schaaf in the so-called local multiplicative case, extend them to the case of a non-local dependence on the bifurcation parameter and to the additive case, both in local and non-local settings.
\end{abstract}

\maketitle

%%%%%%%%%%%%%%%%%%%%%%%%%%%%%%%%%%%%%%%%%%%%%%%%%%%%%%%%%%%%%%%%%%%%%%
%%%%%%%%%%%%%%%%%%%%%%%%%%%%%%%%%%%%%%%%%%%%%%%%%%%%%%%%%%%%%%%%%%%%%%

\section{Introduction}

This paper is devoted to non-existence and uniqueness results for various supercritical semilinear elliptic equations depending on a bifurcation parameter, in a star-shaped domain in $\R^d$. We shall distinguish the \emph{multiplicative case\/} when the equation can be written as
\be{Eqn:MultLoc}
\Delta u +\lambda\,f(u)=0
\ee
and the \emph{additive case\/} for which the equation is
\be{Eqn:AddLoc}
\Delta u+f(u+\mu)=0\;.
\ee
We shall also distinguish two sub-cases for each equation: the \emph{local case\/} when $\lambda$ and $\mu$ are the bifurcation parameters, and the \emph{non-local case\/} when $\lambda$ and $\mu$ are determined by a non-local condition, respectively $$
\lambda\int_\Omega f(u)\;dx=\coupling
$$ and
$$
\int_\Omega f(u+\mu)\;dx=M.
$$
In the \emph{multiplicative non-local} case, the equation is
\be{Eqn:MultNonLoc}
\Delta u +\,\coupling\,\frac{f(u)}{\int_\Omega f(u)\;dx}=0\;.
\ee
In many applications, the term $f(u)/\int_\Omega f(u)\,dx$ is interpreted as a probability measure and $\coupling$ is a coupling parameter. Such a parameter arises from physical constants after a proper adimensionalization. In the \emph{additive non-local} case, the problem to solve is \be{Non-a-Mno}
\Delta u + f(u+\mu)=0\;,\quad M=\int_\Omega f(u+\mu)\;dx\;.
\ee
The parameter $M$ is typically a mass and, in a variational setting, $\mu$ can be interpreted as a Lagrange multiplier associated with the mass constraint, that is, a chemical potential from the point of view of physics.
We shall consider the four problems, \eqref{Eqn:MultLoc}, \eqref{Eqn:AddLoc}, \eqref{Eqn:MultNonLoc} and \eqref{Non-a-Mno}, and prove that if the domain $\Omega$ is star-shaped, with boundary $\partial\Omega$ in $C^{2,\gamma}$, $\gamma\in (0,1)$, and if $f$ is a non-decreasing nonlinearity with supercritical growth at infinity, such that $f(0)>0$ in the case of \eqref{Eqn:MultLoc} or \eqref{Eqn:MultNonLoc}, or such that $f>0$ on $(\bar\mu,\infty)$ and $\lim_{\mu\to\bar\mu}f(\mu)=0$ for some $\bar\mu\in[-\infty,\infty)$ in the case of \eqref{Eqn:AddLoc} or \eqref{Non-a-Mno}, then solutions are unique in $L^\infty\cap H^1_0(\Omega)$ in a certain range of the parameters $\lambda$, $\mu$, $\coupling$ or $M$, while no solution exists for large enough values of the same parameters. Typical nonlinearities are the exponential function $f(u)=e^u$ and the power law nonlinearity $f(u)=(1+u)^p$, for some $p>(d+2)/(d-2)$, $d\geq 3$. In the exponential case, \eqref{Eqn:MultLoc} is the well known \emph{Gelfand equation, cf.\/} \cite{MR0153960}.

Our approach is based on Poho\v{z}aev's estimate, see \cite{MR0192184}, which is obtained by multiplying the equations by $(x\cdot\nabla u)$, integrating over $\Omega$ and then integrating by parts. Also see \cite{MR0002456} for an earlier result based on the local dilation invariance in a linear setting. In this paper, we shall only consider solutions in $L^\infty\cap H^1_0(\Omega)$, which are therefore classical solutions, so that multiplying the equation by $u$ or by $(x\cdot\nabla u)$ is allowed. Some results can be extended to the $H^1_0(\Omega)$ framework, but some care is then required.

\medskip This paper is organized as follows. In Section~\ref{Sec:Mult}, we consider the multiplicative local and non-local bifurcation problems, respectively \eqref{Eqn:MultLoc} and \eqref{Eqn:MultNonLoc}. In Section~\ref{Sec:Add}, we study the additive local and non-local bifurcation problems, respectively \eqref{Eqn:AddLoc} and \eqref{Non-a-Mno}. In all cases, we establish non-existence and uniqueness results, and give some indications on how to construct the branches of solutions, although this is not our main purpose.

\medskip Before giving the details of our results, let us give a brief review of the literature. Concerning \eqref{Eqn:MultLoc}, we primarily refer to the contributions of K. Schmitt in \cite{MR1368678} and R. Schaaf in \cite{MR1785673}, which cover even more general cases than ours and will be discussed more thoroughly later in this section.

The parameter $\lambda$ in \eqref{Eqn:MultLoc} can be seen as a bifurcation parameter. Equation~\eqref{Eqn:MultLoc} is sometimes called a \emph{nonlinear eigenvalue problem.\/} It is well known that for certain values of $\lambda$, multiplicity of solutions can occur, see for instance \cite{MR0340701}. In some cases there are infinitely many positive solutions, even in the radial case, when $\Omega$ is a ball. Radial solutions have been intensively studied. We refer for instance to~\cite{MR1799043} for a review of problems with \emph{positone\/} structure, \emph{i.e.\/} for which $f(0)<0$ and $f$ changes sign once on $\R^+$. A detailed analysis of bifurcation diagrams can be found in \cite{MR1625731,MR1721723}. Also see \cite{MR678562} for earlier and more qualitative results. Positive bounded solutions of such a nonlinear scalar field equation are often called \emph{ground states\/} and can be characterized in many problems as minimizers of a semi-bounded coercive energy functional. They are relevant in many cases of practical interest in physics, chemistry, mathematical biology, \emph{etc.}

When $\Omega$ is a ball, all bounded positive solutions are radial under rather weak conditions on the nonlinearity~$f$, according to \cite{MR544879} and subsequent papers. Lots of efforts have been devoted to uniqueness issues for the solutions of the corresponding ODE and slightly more general problems like quasilinear elliptic ones, see, \emph{e.g.,\/} \cite{MR1462272}. Several other results also cover the case $\Omega=\R^d$, see \cite{MR1803216}. There are also numerous papers in case of more general nonlinearities, including, for instance, functions of $x$, $u$, and $\nabla u$, or more general bifurcation problems than the ones considered in this paper. It is out of the scope of this introduction to review all of them. In a ball, the set of bounded solutions can often be parametrized. The corresponding bifurcation diagrams have the following properties. For nonlinearities with subcritical growth, for instance for $f(u)=(1+u)^p$, $p<(d+2)/(d-2)$, $d\geq 3$, multiple positive solutions may exist when $\lambda$ is positive, small, while for supercritical growths, for example $f(u)=(1+u)^p$ with $p> (d+2)/(d-2)$, $d\geq 3$, or $f(u)=e^u$ and $d=3$, there is one branch of positive solutions which oscillates around some positive, limiting value of $\lambda$ and solutions are unique only for $\lambda$ positive, small. See \cite{MR1793200,DDM,DF,MR678562,MR1625731,MR1721723,MR1186683} for more details.

Another well known fact is that, at least for star-shaped domains, \emph{Poho\-\v{z}aev's method\/} allows to discriminate between super- and subcri\-ti\-cal regimes. This approach has been used mostly to prove the non-existence of non-trivial solutions, see \cite{MR709644,MR855181,MR0333442}, and \cite{MR0192184,MR0002456} for historical references. Such a method is for instance at the basis of the result of \cite{MR709644} on the \emph{Brezis-Nirenberg problem\/}. Also see \cite{MR1793200} and references therein for more details. The identity in \emph{Poho\v{z}aev's method\/} amounts to consider the effect of a dilation on an energy associated to the solution and therefore carries some important information on the problem, see, \emph{e.g.,\/} \cite{MR2437030}. In this context, stereographic projection and connections between euclidean spaces and spheres are natural, as was already noted in \cite{MR1878530} by C.~Bandle and R.~Benguria.

\medskip In this paper we are going to study first the regime corresponding to $\lambda$ small and show that \emph{Poho\v{z}aev's method\/} provides a \emph{uniqueness result} also in cases for which a non-trivial solution exists. The existence of a branch of positive solutions of~\eqref{Eqn:MultLoc} is a widely studied issue, see for instance~\cite{MR0382848,MR0301587}. Also see \cite{MR1368678} for a review, and references therein. As already said, our two basic examples are based on the power law case, $f(u)=(1+u)^p$, and the exponential nonlinearity, $f(u)=e^u$, for which useful informations and additional references can be found in \cite{DDM,MR0340701,MR1648249,MR2093755,MR2135150}. We shall also consider a third example, with a nonlinearity corresponding to the case of \emph{Fermi-Dirac statistics,\/} which behaves like a power law for large, positive values of $u$, and like an exponential function for large, negative values~of~$u$.

The functional framework of bounded solutions and a bootstrap argument imply that we work with classical solutions. Apart from the condition that the domain is star-shaped and satisfies the Poincar\'e inequality, \emph{e.g.,\/} is bounded in one direction, we will assume no other geometrical condition. In the local multiplicative case, several uniqueness results are known for small $\lambda>0$, including in the case of Gelfand's equation, see \cite{MR2010965,MR1785673,MR1368678}. One should note that in the framework of the larger space $H_0^1(\Omega)$, if the boundedness assumption is relaxed, it is not even known if all solutions are radial when $\Omega$ is a ball. The results of~\cite{MR544879} and subsequent papers almost always rely on the assumption that the solutions are continuous or at least bounded on $\overline\Omega$. Notice that, according to \cite{Matano,MR1697245}, even for a ball, it is possible to prescribe a given isolated singularity which is not centered. In \cite{MR1697245}, the case of our two basic examples, $f(u)=e^u$ and $f(u)=(1+u)^p$, with $\frac{d+2}{d-2}<p<\frac{d+1}{d-3}$, $d>3$, has been studied and then generalized to several singularities in \cite{MR1697050}. Also see \cite{MR1071634} for an earlier result. These singularities are in $H^1_0(\Omega)$ and, for a given value of a parameter $\lambda$ set apart from zero, they are located at an \emph{a priori\/} given set of points. Similar problems on manifolds were considered in \cite{MR1134481}.

We refer to \cite{MR1046707,MR958783} for bounds on the solutions to Gelfand's problem, which have been established earlier than uniqueness results but are actually a key tool. Also see \cite{MR1791174} for a more recent contribution. Concerning the uniqueness of the solutions to Gelfand's problem for $d\ge 3$ and $\lambda>0$, small, we refer to \cite{MR2010965,MR1785673,MR1368678}. In the case of a ball, this is even known since the paper of D.D. Joseph and T.S. Lundgren, \cite{MR0340701}, when combined with the symmetry result of \cite{MR544879}.

\medskip The \emph{local multiplicative case\/} corresponding to Problem \eqref{Eqn:MultLoc} is the subject of Section~\ref{Sec:Mult-Local}. The literature on such semilinear elliptic problems and associated biffurcation problems is huge. The results of non-existence of non-trivial solutions are well known, see \cite{MR1780642,MR855181,MR1785673} and references therein. Also see \cite{MR1652967} for extension to systems. Concerning the uniqueness result on non-trivial solutions, the method was apparently discovered independently by several people including F. Mignot and J.-P. Puel, and X. Cabr\'e and P. Majer, \cite{Cabre-Majer}, but it seems that the first published reference on uniqueness results by Rellich-Poho\v{z}aev type estimates is due to K. Schmitt \cite{MR1368678} and later, to R. Schaaf \cite{MR1785673}. A more general result for the multiplicative case has been obtained in \cite{BS-Kyoto} to the price of more intricate reasoning. Numerous papers have been devoted to the understanding of the role of the geometry and they extend the standard results, mostly the non-existence results, to the case of non strictly star-shaped domains: see for instance \cite{MR1780642,MR855181,MR1785673} and several papers of J. McGough \emph{et al.}, see \cite{MR2010965,MR2135243,MR1791174}, which are, as far as we know, the most up-to-date results on such issues.

\medskip As already mentioned above, Problem \eqref{Eqn:MultLoc} has been studied by K.~Schmitt in \cite{MR1368678} and R. Schaaf in \cite{MR1785673}. In \cite[Theorem 2.6.7]{MR1368678}, it is proved that if one replaces $f(u)$ in \eqref{Eqn:MultLoc} by a more general function $f(x,u)$ in $C^2(\overline\Omega\times\R^+)$ satisfying
\begin{eqnarray*}
&&(i)\, f(x,u)>0, \; f_u(x,u)>0, \; u\geq 0, \; x\in \overline{\Omega}\;,\\
&&(ii)\, \limsup_{u\rightarrow\infty} \sup_{x\in \overline{\Omega}} \frac{2\,d\,F(x,u)}{(d-2)\,u\,f(x,u)}<1\;, \\
&&(iii)\, \left[ \nabla_x F(x,u+1)-\nabla_x F(x,u)-u\,\nabla_x f(x,u) \right] \cdot x \leq 0\quad\mbox{for}\; u\gg 1\,,\; x\in\overline\Omega\;,
\end{eqnarray*}
then uniqueness holds for a star-shaped domain $\Omega$. A survey on the existence and continuation results for linear and superlinear (sub- and supercritical) growth of the nonlinear term $f$ in \eqref{Eqn:MultLoc} can also be found in \cite{MR1368678}, as well as a study of the influence of the geometry, topology and dimension of the domain, which is of interest for our purpose.

In \cite{MR1785673}, R. Schaaf studies uniqueness results for the semilinear elliptic problem~\eqref{Eqn:MultLoc} under the asymptotic condition $\limsup_{u\rightarrow \infty} \tfrac{F(u)}{u\,f(u)}< \tfrac{1}{2}-M(\Omega)$ where $M(\Omega)=1/d$ for star-shaped domains. In general $M(\Omega)$ is some number in the interval $(0,1/d]$. In the autonomous case, the above asymptotic condition is equivalent to the assumption (ii) made by K. Schmitt in \cite{MR1368678} or to our assumption~\eqref{Sc1}, to be found below. Our contribution to the question of the uniqueness for~\eqref{Eqn:MultLoc} relies on a simplification of the proof in \cite{MR1785673,MR1368678}.

\medskip Imposing a \emph{non-local constraint} dramatically changes the picture. For instance, in case of Maxwell-Boltzmann statistics, $f(u)=e^u$, in a ball of $\R^2$, the solution of \eqref{Eqn:MultLoc} has two solutions for any $\lambda\in (0,\lambda_*)$ and no solution for $\lambda>\lambda_*$, while uniqueness holds in \eqref{Eqn:MultNonLoc} in terms of $M$, for any $M$ for which a solution exists, see \cite{MR2043943,MR0340701}. Non-local constraints are motivated by considerations arising from physics. In the case of the exponential nonlinearity with a mass normalization constraint, a considerable effort has been done in the two-dimensional case for understanding the statistical properties of the so-called \emph{Onsager solutions} of the Euler equation, see \cite{MR1145596,MR1362165,MR2068307}. The same model, but rather in dimension $d=3$, is relevant in astrophysical models for systems of gravitating particles, see~\cite{BS-Kyoto}. Other standard examples are the polytropic distributions, with $f(u)=u^p$, and Bose-Einstein or Fermi-Dirac distributions which result in nonlinearities involving special functions. Existence and non-existence results were obtained for instance in \cite{MR2043943} and \cite{MR2134461,MR2136979}, respectively for Maxwell-Boltzmann and Fermi-Dirac statistics.

An evolution model compatible with Fermi-Dirac statistics and the convergence of its solutions towards steady states has been thoroughly examined in \cite{MR2099972}, while the steady state problem was considered by R. Sta\'nczy in \cite{MR2134461, MR2136979,Stanczy08}. See~\cite{MR1934935} and references therein for a model improved with respect to thermodynamics, \cite{MR2136979} and references therein for more elaborate models, and \cite{CSR} for a derivation of an evolution equation involving a mean field term, which also provides a relevant, stationary model studied in \cite{BS-Kyoto, MR2295189}. Also see \cite{MR2092680, DMOS} for an alternative, phenomenological derivation of drift-diffusion equations and their stationary counterparts, and \cite{Stanczy08} for the existence of radial solutions by fixed point methods in weighted function spaces, under nonlocal constraints. The case of a decoupled, external potential goes back to the work of Smoluchowski, see \cite{MR1781665,Smo}. For this reason, the evolution model is often referred to as the \emph{Smoluchowski--Poisson equation.}

Our purpose is not to study the above mentioned evolution equations, but only to emphasize that for the corresponding steady states, non-local constraints are very natural, since they correspond to quantities which are conserved along evolution. Hence, to identify the asymptotic state of the solutions to the evolution equation, we have to solve a semilinear elliptic equation with a non-local constraint, which corresponds, for instance, to mass conservation.

%%%%%%%%%%%%%%%%%%%%%%%%%%%%%%%%%%%%%%%%%%%%%%%%%%%%%%%%%%%%%%%%%%%%%%
%%%%%%%%%%%%%%%%%%%%%%%%%%%%%%%%%%%%%%%%%%%%%%%%%%%%%%%%%%%%%%%%%%%%%%
\section{The multiplicative case}\label{Sec:Mult}

%%%%%%%%%%%%%%%%%%%%%%%%%%%%%%%%%%%%%%%%%%%%%%%%%%%%%%%%%%%%%%%%%%%%%%
\subsection{The local bifurcation problem}\label{Sec:Mult-Local}

We consider Problem~\eqref{Eqn:MultLoc} on a domain $\Omega$ in $\R^d$. Our first assumption is the geometrical condition that a Poincar\'e inequality holds:
\be{Ineq:Poincare}
\int_\Omega|u|^2\,dx\leq C_{\rm P}\int_\Omega|\nabla u|^2\,dx
\ee
for any $u\in H^1_0(\Omega)$ and some positive constant $C_{\rm P}>0$. Such an inequality holds for instance if $\Omega$ is bounded in one direction. See \cite[Proposition 2.1]{MR1680893} for more details. Inequality \eqref{Ineq:Poincare} is called \emph{Friedrichs' inequality} in some areas of analysis (see \cite{MR1512401,Poincare1887} for historical references; we also refer to \cite{Jakszto}). We shall further require that
\be{PosOptFn}
\exists\;u\in H^1_0(\Omega)\;\mbox{\sl such that}\quad u>0\quad\mbox{\sl and}\quad\int_\Omega|u|^2\,dx=C_{\rm P}\int_\Omega|\nabla u|^2\,dx\;.
\ee
This is straightforward in some cases, for instance if $\Omega$ is bounded, simply connected, with a Lipschitz boundary, or if $\Omega$ is unbounded, simply connected and such that the embedding $H^1_0(\Omega)\hookrightarrow L^2(\Omega)$ is compact. For such a compactness property, see for instance \cite[Theorem 2.8]{MR0312241} and \cite[Theorems 6.16 and 6.19]{MR2424078}.

The goal of this section is to state a non-existence result for large values of~$\lambda$ and give sufficient conditions on $f\geq 0$ such that, for some $\lambda_0>0$, Equation~\eqref{Eqn:MultLoc} has a unique solution in $L^\infty\cap H^1_0(\Omega)$ for any $\lambda\in(0,\lambda_0)$. We assume that $f$ is of class $C^2$. By standard elliptic bootstraping arguments, a bounded solution is then a classical one.

Next we assume that {\em for some $\lambda_*>0$, there exists a branch of positive minimal solutions $(\lambda,u_\lambda)_{\lambda\in(0,\lambda_*)}$ originating from~$(0,0)$ and such that}
\be{H1}
\lim_{\lambda\to 0_+}\left(\|u_\lambda\|_{L^\infty(\Omega)}+\|\nabla u_\lambda\|_{L^\infty(\Omega)}\right)=0\;.
\ee
Sufficient conditions for such a property to hold can be found in various papers. We can for instance quote the following result.
%---------------------------------------------------------------------
\begin{lemma}\label{Lem:existence} Assume that $\Omega$ is bounded with smooth, {\rm i.e.} $C^{2,\gamma}$ for some $\gamma\in (0,1)$, boundary, $f\in C^2$ is positive on $[0,\infty)$ and $\inf_{u>0}f(u)/u>0$. Then \eqref{H1} holds. \end{lemma}
%---------------------------------------------------------------------
We refer for instance to \cite{MR1368678} for a proof. The solutions satisfying \eqref{H1} can be characterized as a branch of minimal solutions, using sub- and super-solutions. Although this is standard, for the sake of completeness let us state a non-existence result for values of the parameter $\lambda$ large enough.
%---------------------------------------------------------------------
\begin{proposition}\label{Prop:Explosion} Assume that \eqref{Ineq:Poincare} and \eqref{PosOptFn} hold. If $\Lambda:=\inf_{u>0}f(u)/u>0$, then there exists $\lambda_*>0$ such that \eqref{Eqn:MultLoc} has no non-trivial nonnegative solution in $H^1_0(\Omega)$ if $\lambda>\lambda_*$. \end{proposition}
%---------------------------------------------------------------------
\noindent
The lowest possible value of $\lambda_*$ is usually called the \emph{critical explosion parameter.\/} \medskip

\begin{proof} Let $\varphi_1$ be a positive eigenfunction associated with the first eigenvalue $\lambda_1=1/C_{\rm P}$ of $-\Delta$ in $H^1_0(\Omega)$:
\[
-\Delta\varphi_1=\lambda_1\,\varphi_1\;.
\]
By multiplying this equation by $u$ and \eqref{Eqn:MultLoc} by $\varphi_1$, we get
\[
\lambda_1\int_\Omega u\,\varphi_1\;dx=\int_\Omega \nabla u\cdot\nabla\varphi_1\;dx=\lambda \int_\Omega f(u)\,\varphi_1\;dx\geq \Lambda\,\lambda\int_\Omega u\,\varphi_1\;dx\;,
\]
thus proving that there are no non-trivial nonnegative solutions if $\lambda>\lambda_1/ \Lambda$. \end{proof}

Next we present a simplified version of the proof of a uniqueness result stated in \cite{MR1785673}, under slightly more restrictive hypotheses. We assume that $d\geq 3$ and that~$f$ has a supercritical growth at infinity, \emph{i.e.,\/} $f$ is such that
\be{Sc1}
\limsup_{u\rightarrow\infty}\frac{F(u)}{u\,f(u)}=\eta<\frac{d-2}{2\,d}\,,
\ee
where $F(u):=\int_0^uf(s)\,ds$. Notice that, in Proposition~\ref{Prop:Explosion}, $ \Lambda>0$ if \eqref{Sc1} holds and if we assume that $f$ is positive.
%---------------------------------------------------------------------
\begin{theorem}\label{Thm:Multi0} Assume that $\Omega$ is a bounded star-shaped domain in $\R^d, d\ge 3$, with $C^{2,\gamma}$ boundary, such that \eqref{Ineq:Poincare} holds for some \hbox{$C_{\rm P}\!>\!0$}. If $f(z)$ is positive for large values of $z$, of class $C^2$ and satisfies \eqref{H1} and \eqref{Sc1}, then there exists a positive constant $\lambda_0$ such that Equation~\eqref{Eqn:MultLoc} has at most one solution in $L^\infty\cap H^1_0(\Omega)$ for any $\lambda\in(0,\lambda_0)$. \end{theorem}
%---------------------------------------------------------------------
\begin{proof} We follow the lines of the proof of \cite{MR1785673} with some minor simplifications. Up to a translation, we can assume that $\Omega$ is star-shaped with respect to the origin. Assume that \eqref{Eqn:MultLoc} has two solutions, $u$ and $u+v$. With no restriction, we can assume that $u$ is a minimal solution and satisfies \eqref{H1}. As a consequence, $v$ is nonnegative and satisfies
\be{Eqn:Diff}
\Delta v+\lambda\,\big[f(u+v)-f(u)\big]=0\;.
\ee
If we multiply \eqref{Eqn:Diff} by $v$ and integrate with respect to $x\in\Omega$, we get
\be{Eqn:Energy-v}
\int_\Omega|\nabla v|^2\,dx=\lambda\int_\Omega v\,\big[f(u+v)-f(u)\big]\,dx\;.
\ee
Multiply \eqref{Eqn:Diff} by $x\cdot\nabla v$ and integrate with respect to $x\in\Omega$ to get
\begin{multline}\label{Po1}
\frac{d-2}2\int_\Omega|\nabla v|^2\,dx+\frac 12\int_{\partial\Omega}|\nabla v|^2(x\cdot\nu(x))\,d\sigma\\
=d\,\lambda\int_\Omega\left[F(u+v)-F(u)-F'(u)\,v\right]\,dx \\
+\;\lambda\int_\Omega(x\cdot\nabla u)\left[f(u+v)-f(u)-f'(u)\,v\right]\,dx
\end{multline}
where $d\sigma$ is the measure induced by Lebesgue's measure on $\partial\Omega$. Recall that $F$ is a primitive of $f$ such that $F(0)=0$. Take $\eta_1\in (\eta,(d-2)/(2\,d))$ where~$\eta$ is defined in Assumption \eqref{Sc1}. Since $u=u_\lambda$ is a minimal solution and therefore uniformly small as $\lambda\to 0_+$, for any $\varepsilon>0$, we obtain $|x\cdot\nabla u|\le\varepsilon$ for any $x\in\Omega$, provided $\lambda>0$ is small enough. Define $h_\varepsilon$ by
\begin{eqnarray*}
h_\varepsilon(u,v):=d\,\big[F(u+v)-F(u)-F'(u)\,v\big]+\varepsilon\,\left|f(u+v)-f(u)-f'(u)\,v \right| \!\qquad&&\\
-\;d\,\eta_1\, v\,\big[f(u+v)-f(u)\big]\,.&&
\end{eqnarray*}
Because of the smoothness of $f$ and by Assumption \eqref{Sc1}, the function $h_\varepsilon(u,v)/v^2$ is bounded from above by some constant $H$, uniformly in $\varepsilon>0$, small enough. By the assumption of star-shapedeness of the domain~$\Omega$, $x\cdot\nu(x)\geq 0$ for any $x\in\partial\Omega$. From \eqref{Eqn:Energy-v} and \eqref{Po1}, it follows that
\[
\frac{d-2}2\int_\Omega|\nabla v|^2\,dx\leq d\,\lambda\,H\int_\Omega|v|^2\,dx + d \,\eta_1\int_\Omega|\nabla v|^2\,dx\;.
\]
Due to the Poincar\'e inequality \eqref{Ineq:Poincare}, the condition
\[
\lambda<\frac 1{C_{\rm P}\,H}\,\left(\frac{d-2}{2\,d}-\eta_1\right)
\]
implies $v=0$ and the uniqueness follows. \end{proof}

%---------------------------------------------------------------------
\medskip\noindent{\bf Examples}
\begin{enumerate}
\item{If $f(u)=e^u$, Condition \eqref{Sc1} is always satisfied. Notice that if $d=2$ and $\Omega$ is a ball, the uniqueness result is wrong, see \cite{MR0340701}.}
\item{If $f(u)=(1+u)^p$, $d\geq 3$, Condition~\eqref{Sc1} holds if and only if $p>\frac{d+2}{d-2}$. Also see \cite{MR0340701} for more details. Similarly in the same range of parameters for $f(u)=u^p$ we only get the trivial, zero solution.}
\item{The Fermi-Dirac distribution
\begin{equation}\label{Fermi}
f(u)=f_\delta(u):=\int_{0}^{\infty}\frac{t^{\delta}}{1+e^{t-u}}\;dt
\end{equation}
behaves like $\frac{1}{\delta+1}u^{\delta+1}$ as $u\to\infty$. Condition~\eqref{Sc1} holds if and only if $\delta+1>(d+2)/(d-2)$. The physically relevant examples require that $\delta=d/2-1$, that is $d>2\,(1+\sqrt{2})\approx4.83$. For more properties of these functions see, \emph{e.g.,\/}~\cite{MR2099972,MR2143357}.}
\end{enumerate}
%---------------------------------------------------------------------

%%%%%%%%%%%%%%%%%%%%%%%%%%%%%%%%%%%%%%%%%%%%%%%%%%%%%%%%%%%%%%%%%%%%%%
\subsection{The non-local bifurcation problem}\label{Sec:Mult-NonLocal}

In this section we address, in $L^\infty\cap H^1_0(\Omega)$, the non-local boundary value problem~\eqref{Eqn:MultNonLoc} with parameter $\coupling>0$. Here $\Omega$ is a bounded domain in $\R^d$, $d\ge 3$, with $C^1$ boundary.

\bigskip We start with a non-existence result. Computations are similar to the ones of Section~\ref{Sec:Mult-Local} and rely on Poho\v{z}aev's method. First multiply \eqref{Eqn:MultNonLoc} by $u$ to get
\be{Eqn:Energy-MultNonLocal}
\int_\Omega|\nabla u|^2\;dx=\coupling\,\frac{\int_\Omega u\,f(u)\;dx}{\int_\Omega f(u)\;dx}\;.
\end{equation}
Multiplying \eqref{Eqn:MultNonLoc} by $(x\cdot\nabla u)$, we also get
\begin{equation}\label{mn0}
\frac{d-2}2\int_\Omega|\nabla u|^2\,dx+\frac 12\int_{\partial\Omega}|\nabla u|^2\,(x\cdot\nu)\,d\sigma=d\,\coupling\,\frac{\int_\Omega F(u)\;dx}{\int_\Omega f(u)\;dx}
\end{equation}
where $F$ is the primitive of $f$ chosen so that $F(0)=0$ and $d\sigma$ is the measure induced by Lebesgue's measure on $\partial\Omega$. A simple integration of \eqref{Eqn:MultNonLoc} gives
\[
\coupling=-\int_\Omega\Delta u\,dx=-\int_{\partial\Omega}\nabla u\cdot \nu\;d\sigma\;.
\]
By the Cauchy-Schwarz inequality,
\[
\coupling^2=\left(\int_{\partial\Omega}\nabla u\cdot \nu\;d\sigma\right)^2\leq |\partial\Omega|\int_{\partial\Omega}|\nabla u\cdot \nu|^2\,d\sigma=|\partial\Omega|\int_{\partial\Omega}|\nabla u|^2\,d\sigma\;,
\]
where the last equality holds because of the boundary conditions. Assume that $\Omega$ is strictly star-shaped with respect to the origin:
\be{star-shapedness}
\alpha:=\inf_{x\in\partial\Omega}(x\cdot\nu(x))>0\;.
\ee
Because of the invariance by translation of the problem, this is equivalent to assume that $\Omega$ is strictly star-shaped with respect to any other point in $\R^d$. Hence
\[
\int_{\partial\Omega}|\nabla u|^2\,(x\cdot\nu)\;d\sigma\geq \alpha \, \int_{\partial\Omega}|\nabla u|^2\;d\sigma \geq\frac{\alpha\,\coupling^2}{|\partial \Omega|}\;.
\]
Collecting this estimate with \eqref{Eqn:Energy-MultNonLocal} and \eqref{mn0}, we obtain
\[
\int_\Omega\Big[2\,d\,F(u)-(d-2)\,u\,f(u)\Big]\,dx \geq \frac{\alpha\,\coupling}{|\partial\Omega|} \int_\Omega f(u)\;dx\;.
\]
As a straightforward consequence, we obtain the following result.
%---------------------------------------------------------------------
\begin{theorem}\label{Thm:Non-mult} Assume that $\Omega$ is a bounded domain in $\R^d$, $d\ge 3$, with $C^1$ boundary satisfying~\eqref{star-shapedness} for some $\alpha>0$. If $f$ is a $C^1$ function such that for some $C>0$,
\be{Non-mul}
2\,d\,F(u)\leq (d-2)\,u\,f(u)+C\,f(u)
\ee
for any $u\geq 0$, then \eqref{Eqn:MultNonLoc} has no solution in $L^\infty\cap H^1_0(\Omega)$ if $\coupling>C\,|\partial\Omega|/\alpha$. \end{theorem}
%---------------------------------------------------------------------

\noindent Standard examples, for which Condition \eqref{Non-mul} is satisfied, are:
\begin{enumerate}
\item\emph{Exponential case:\/} $f(u)=e^{u}$ with $C=2\,d$, \emph{cf.\/} \cite{MR2043943}. A sharper estimate can be easily achieved as follows. The function $h(u):=C\,e^u+(d-2)\,u\,e^u-2\,d\,(e^u-1)$ is nonnegative if $C$ is such that $0=h'(u)=h(u)$ for some $u\ge 0$. After eliminating~$u$, we find
\be{ExpSharp}
C=d+2+(d-2)\,\log\left(\frac{d-2}{2\,d}\right)\;.
\ee
\item\emph{Pure power law case:\/} If $f(u)=u^p$, the result holds with $p\geq \frac{d+2}{d-2}$ and $C=0$, \emph{cf.\/} \cite{MR0153960,MR1428686}. There are no non-trivial solutions.
\item\emph{Power law case:\/} If $f(u)=(1+u)^p$ with $p\geq \frac{d+2}{d-2}$, then \eqref{Non-mul} holds with $C=d-2$.
\end{enumerate}

\bigskip Uniqueness results in the non-local case follow from Section~\ref{Sec:Mult-Local}, when the coupling constant $\coupling$ is positive, small. In case of nonlinearities of exponential type, as far as we know, uniqueness results were guaranteed only under some additional assumptions, see \cite{MR1242651,MR1268074}. We are now going to extend such uniqueness results to more general nonlinearities satisfying \eqref{H1} and \eqref{Sc1} by comparing Problems \eqref{Eqn:MultLoc} and~\eqref{Eqn:MultNonLoc}.

Denote by $u_\lambda$ the solutions of \eqref{Eqn:MultLoc}. For $\lambda>0$, small, a branch of solutions of~\eqref{Eqn:MultNonLoc} can be parametrized by $\lambda\mapsto\left(\textstyle\coupling(\lambda):=\lambda\int_\Omega f(u_\lambda)\,dx\,,\;u_\lambda\right)$. Reciprocally, if $\Omega$ is bounded and
\[
0<\beta:=\inf_{u\geq 0}f(u)\;,
\]
then any solution $u\in L^\infty\cap H^1_0(\Omega)$ of~\eqref{Eqn:MultNonLoc} is also a solution of \eqref{Eqn:MultLoc} with
\[
\lambda=\frac\coupling{\int_\Omega f(u)\;dx}\leq\frac\coupling{\beta\,|\Omega|}\;.
\]
This implies that $\lambda$ is small for small $\coupling$ and, as a consequence, for small values of~$\coupling$, all solutions to \eqref{Eqn:MultNonLoc} are located somewhere on the local branch originating from~$(0,0)$. Moreover, as $\coupling\to 0_+$, the solution of \eqref{Eqn:MultNonLoc} also converges to~$(0,0)$. To prove the uniqueness in $L^\infty\cap H^1_0(\Omega)$ of the solutions of~\eqref{Eqn:MultNonLoc}, it is therefore sufficient to establish the monotonicity of $\lambda\mapsto\coupling(\lambda)$ for small values of $\lambda$. Assume that
\be{Monotonicity}
\mbox{\em $f(0)>0$ and $f$ is monotone non-decreasing on $\R^+$.}
\ee
Under this assumption, we observe that $\beta=f(0)$.

Let $u_1$ and $u_2$ be two solutions of \eqref{Eqn:MultLoc} with $\lambda_1<\lambda_2$ and let $v:=u_2-u_1$. Then for some function $\theta$ on $\Omega$, with values in $[0,1]$, we have
\[
-\Delta v-\lambda_1\,f'(u_1+\theta\,v)\,v=(\lambda_2-\lambda_1)\,f(u_2)\geq 0\;,
\]
so that, by the Maximum Principle, $v$ is nonnegative. Notice indeed that for $\lambda_2$ small enough, $u_1$ and $u_2$ are uniformly small since they lie on the local branch, close to the point $(0,0)$ and therefore $\lambda_1\,f'(u_1+\theta\,v)<1/C_{\rm P}$. It follows that
\[
\int_\Omega f(u_2)\;dx=\int_\Omega f(u_1+v)\;dx\geq\int_\Omega f(u_1)\;dx\;,
\]
thus proving that $\coupling(\lambda_2)=\lambda_2\int_\Omega f(u_2)\,dx>\lambda_1\int_\Omega f(u_1)\,dx=\coupling(\lambda_1)$.
%---------------------------------------------------------------------
\begin{corollary}\label{Cor:Multi1} Under the assumptions of Theorem~\ref{Thm:Multi0}, if moreover $f$ satisfies~\eqref{Monotonicity}, then there exists a positive constant $\coupling_0$ such that Equation~\eqref{Eqn:MultNonLoc} has at most one solution in $L^\infty\cap H^1_0(\Omega)$ for any $\coupling\in (0,\coupling_0)$.\end{corollary}
%---------------------------------------------------------------------

%%%%%%%%%%%%%%%%%%%%%%%%%%%%%%%%%%%%%%%%%%%%%%%%%%%%%%%%%%%%%%%%%%%%%%
%%%%%%%%%%%%%%%%%%%%%%%%%%%%%%%%%%%%%%%%%%%%%%%%%%%%%%%%%%%%%%%%%%%%%%
\section{The additive case}\label{Sec:Add}

%%%%%%%%%%%%%%%%%%%%%%%%%%%%%%%%%%%%%%%%%%%%%%%%%%%%%%%%%%%%%%%%%%%%%%
\subsection{The local bifurcation problem}\label{Sec:Add-Local}

Consider in $L^\infty\cap H^1_0(\Omega)$ the equation \eqref{Eqn:AddLoc}. In the two standard examples of this paper the problem can be reduced to \eqref{Eqn:MultLoc} as follows.
\begin{enumerate}
\item\emph{Exponential case:\/} If $f(u)=e^{u}$, \eqref{Eqn:AddLoc} is equivalent to \eqref{Eqn:MultLoc} with $\lambda=e^\mu$ and the limit $\lambda\to 0_+$ corresponds to $\mu\to -\infty$.
\item\emph{Power law case:\/} If $f(u)=(1+u)^p$, \eqref{Eqn:AddLoc} is equivalent to \eqref{Eqn:MultLoc} with $\lambda=(1+\mu)^{p-1}$ and the limit $\lambda\to 0_+$ corresponds to $\mu\to -1_+$. If $u$ is solution of $\Delta u + (1+u+\mu)^p=0$, one can indeed observe that $v$ such that $1+u+\mu=(1+\mu)(1+v)$ solves $\Delta v+ \lambda\,(1+v)^p=0$ with $\lambda=(1+\mu)^{p-1}.$
\end{enumerate}

Equation \eqref{Eqn:AddLoc} is however not completely equivalent to \eqref{Eqn:MultLoc}. To obtain a non-existence result for large values of $\mu$, we impose the assumption that reads
\be{Cdt:Superlinear}
\lim_{u\to \infty}\frac{f(u)}u=+\infty\;.
\ee
%---------------------------------------------------------------------
\begin{proposition}\label{Prop:Explosion2} Assume that \eqref{Ineq:Poincare}, \eqref{PosOptFn} and \eqref{Cdt:Superlinear} hold. There exists $\mu_*>0$ such that \eqref{Eqn:AddLoc} has no positive, bounded solution in $H^1_0(\Omega)$ if $\mu>\mu_*$. \end{proposition}
%---------------------------------------------------------------------
\begin{proof} The proof is similar to the one of Proposition \ref{Prop:Explosion}. Let $\varphi_1$ be a positive eigenfunction associated with the first eigenvalue $\lambda_1=1/C_{\rm P}$ of $-\Delta$ in $H^1_0(\Omega)$. For any~$\mu\geq 0$,
\[
\lambda_1\int_\Omega u\,\varphi_1\;dx=\int_\Omega f(u+\mu)\,\varphi_1\;dx\geq \Lambda(\mu)\int_\Omega (u+\mu)\,\varphi_1\;dx\ge \Lambda(\mu)\int_\Omega u\,\varphi_1\;dx\;,
\]
where $\Lambda(\mu):=\inf_{s\geq\mu}f(s)/s$, thus proving that there are no nonnegative solutions if $\Lambda(\mu)>\lambda_1$. \end{proof}

\medskip Let us make a few comments on the existence of a branch of solutions, although this is out of the main scope of this paper. Let $f$ be a positive function of class $C^2$ on $(\bar\mu,\infty)$, for some $\bar\mu\in [-\infty,\infty)$, with $\lim_{\mu\to\bar\mu_+}f(\mu)=0$. We shall assume that there is a branch of minimal solutions $(\mu ,u_\mu)$ originating from $(\bar\mu,0)$ and such that
\be{H1_}
\lim_{\mu\to \bar\mu }\left(\|u_\mu \|_{L^\infty(\Omega)}+\|\nabla u_\mu\|_{L^\infty(\Omega)}\right)=0\;.
\ee
This can be guaranteed if $\Omega$ is bounded and if we additionally require that the function $f$ is increasing, as in \cite{MR2134461} for the Fermi-Dirac model. This is also true for exponential and power-like nonlinearities. At least at a formal level, this can easily be understood by taking $\zeta = f'(\mu)$ as a bifurcation parameter. A solution of \eqref{Eqn:AddLoc} is then a zero of $F(\zeta,u)= u-\left(-\Delta \right)^{-1} f(u+(f')^{-1}(\zeta))$ and it is therefore easy to find a branch issued from $(\zeta,u)=(0,0)$ by applying the implicit function theorem at $(\zeta,u)=(0,0)$ with $F(0,0)=0$, even if $\bar\mu=-\infty$. Using comparison arguments, one can prove that this branch is a branch of minimal solutions.

\medskip We shall now address the uniqueness issues. We assume that \eqref{Sc1} holds:
\[
\forall\;\eta_1\in\left(\eta,\frac{d-2}{2\,d}\right)\;,\quad\limsup_{u\to\infty}\frac{F(u)-\eta_1\,u\,f(u)}{u\,f(u)}=\eta-\eta_1<0\;.
\]
As a consequence, for any $\mu>\bar\mu$,
\[
F(v+\mu)-F(\mu)-F'(\mu)\,v-\eta_1 v\,\big[f(v+\mu)-f(\mu)\big]
\]
is negative for large $v$, and the function ${\mathcal H}(v,\mu,\eta_1)$ defined by
\[
v^2\,{\mathcal H}(v,\mu,\eta_1)=F(v+\mu)-F(\mu)-F'(\mu)\,v-\eta_1\, v\,\big[f(v+\mu)-f(\mu)\big]
\]
achieves a maximum for some finite value of $v$. With $H(\mu,\eta_1)=\sup_{v>0}{\mathcal H}(v,\mu,\eta_1)$, we have
\begin{equation}\label{Hv2}
F(v+\mu)-F(\mu)-F'(\mu)\,v-\eta_1\, v\,\big[f(v+\mu)-f(\mu)\big]\leq H(\mu,\eta_1)\,v^2\,.
\end{equation}
Next we assume that, for some $\eta_1\in\left(\eta,\frac{d-2}{2\,d}\right)$, we have
\be{Cdt:AddLoc}
C_{\rm P}\,H(\mu,\eta_1)<\frac{d-2}{2\,d}-\eta_1\;,
\ee
where $C_{\rm P}$ is the Poincar\'e constant. This condition is non-trivial. It relates $H(\mu,\eta_1)$, a quantity attached to the nonlinearity, to $C_{\rm P}$ which has to do only with $\Omega$. It is satisfied for all our basic examples.
\begin{enumerate}
\item\emph{Exponential case:\/} If $f(u)=e^{u}$, we take $\mu$ negative, with $|\mu|$ big enough. Indeed, using first the homogeneity, one obtains ${\mathcal H}(v,\mu,\eta_1)=e^\mu\,{\mathcal H}(v,0,\eta_1)$. Since $\lim_{v\to 0_+}{\mathcal H}(v,0,\eta_1)=(1-2\,\eta_1)/2$ and ${\mathcal H}(v,0,\eta_1)$ becomes negative as $v\to +\infty$, as a function of $v\in\R^+$ ${\mathcal H}(v,0,\eta_1)$ admits a maximum value. To get a more explicit bound, we take a Taylor expansion at second order, namely $e^{\theta\,v}(1-2\,\eta_1-\eta_1\,\theta\,v)/2$ for some intermediate number $\theta\in (0,1)$. An upper bound is given by $\eta_1\,e^{1/\eta_1-3}/2$, which corresponds to the above expression evaluated at $\theta\,v=1/\eta_1-3$. According to \eqref{Sc1}, $\eta=0$: taking $\eta_1$ small enough guarantees \eqref{Cdt:AddLoc}.
\item\emph{Power law case:\/} If $f(u)=(1+u)^p$, we have ${\mathcal H}(v,\mu,\eta_1)=(1+\mu)^{p+1}{\mathcal H}(w,0,\eta_1)$ where $w=v/(\mu+1)$. Since $\lim_{v\to 0_+}{\mathcal H}(v,0,\eta_1)=p\,(1-2\,\eta_1)/2$ and ${\mathcal H}(v,0,\eta_1)$ becomes negative as $v\to +\infty$, ${\mathcal H}$ achieves a positive maximum.
\item\emph{Fermi-Dirac distribution case:\/} If $f(u)=f_{d/2-1}(u)$, we observe that
\begin{equation} \label{Sc2}
\limsup_{u\rightarrow \infty} \frac{f'(u)}{u\,f''(u)+2\,f'(u)}=\eta< \frac{d-2}{2\,d}
\end{equation}
if $d>2\,(1+\sqrt{2})$, which is stronger than Assumption \eqref{Sc1}, as can easily be recovered by integrating $f'(u)-\eta\,\big[u\,f''(u)+2\,f'(u)\big]$ twice, for large values of $u$. Take $\eta_1\in (\eta,(d-2)/(2\,d))$. A Taylor expansion shows that
\begin{eqnarray*}
{\mathcal H}(v,\mu,\eta_1)&=&f'(u)-\eta_1\,\big(u\,f''(u)+2\,f'(u)\big)+\mu\,\eta_1\,f''(u)\\
&=&a\left[f'(u)-\tfrac{\eta+\eta_1}2\,\big(u\,f''(u)+2\,f'(u)\big)\right]+(\mu-b\,u)\,\eta_1\,f''(u)
\end{eqnarray*}
with $a=\tfrac {1-2\,\eta_1}{1-\eta-\eta_1}$, $b=\tfrac{\eta_1-\eta}{2\,\eta_1(1-\eta-\eta_1)}$ and $u=\mu+\theta\,v$ for some $\theta\in(0,1)$. Both terms in the above right hand side are negative for $u$ large enough, which proves the existence of a constant $H(\mu,\eta_1)$ such that \eqref{Hv2} holds. Notice that by \cite[Appendix]{MR2143357}, $f$ and its derivatives behave like exponentials for $u<0$, $|u|$ large. Under the additional assumption $d\ge 6$, a tedious but elementary computation shows that, as $\mu\to -\infty$, the maximum of
\[
u\mapsto a\left[f'(u)-\tfrac{\eta+\eta_1}2\,\big(u\,f''(u)+2\,f'(u)\big)\right]+(\mu-b\,u)\,\eta_1\,f''(u)
\]
is achieved at some $u=o(\mu)$, which proves that \hbox{$\lim_{\mu\to-\infty}H(\mu,\eta_1)=0$}. Moreover, for any $d>2(1+\sqrt{2})$ one can still show that this maximum value behaves like $\exp(\mu)$ and thus can be made arbitrarily small for negative $\mu$ with $|\mu|$ large enough.
\end{enumerate}

\medskip Assume that \eqref{Eqn:AddLoc} has two solutions, $u$ and $u+v$, with $v\ge 0$, and let us write the equation for the difference $v$ as
\be{Eqn:Diff_}
\Delta v+f(u+v+\mu)-f(u+\mu)=0\;.
\ee
The method is the same as in Section~\ref{Sec:Mult}. Multiply \eqref{Eqn:Diff_} by $x\cdot\nabla v$ and integrate with respect to $x\in\Omega$. If $F$ is a primitive of $f$ such that $F(\bar\mu)=0$, then
\begin{multline*}
\frac{d-2}2\int_\Omega|\nabla v|^2\,dx+\frac 12\int_{\partial\Omega}|\nabla v|^2(x\cdot\nu(x))\,d\sigma\\
=d\,\int_\Omega\big[F(u+v+\mu)-F(u+\mu)-F'(u+\mu)\,v\big]\,dx \\
+\int_\Omega(x\cdot\nabla u)\big[f(u+v+\mu)-f(u+\mu)-f'(u+\mu)\,v\big]\,dx\;.
\end{multline*}
Assume that \eqref{Cdt:AddLoc} holds for some $\eta_1$. If $\Omega$ is bounded, $|x\cdot\nabla u|$ is uniformly small as $\mu\to\bar\mu_+$, and we may assume that for any \hbox{$\varepsilon>0$}, arbitrarily small, there exists $\mu_0>\bar\mu$, sufficiently close to $\bar\mu$ (that is, $\mu_0-\bar\mu>0$, small if $\bar\mu>-\infty$, or $\mu_0<0$, $|\mu_0|$ big enough if $\bar\mu=-\infty$), such that $|x\cdot\nabla u|\le\varepsilon$ for any $x\in \Omega$ if $\mu\in(\bar\mu,\mu_0)$. Next we define
\begin{multline*}\label{hfu}
h_\varepsilon(v):=d\,\big[F(z+v)-F(z)-F'(z)\,v\big]+\varepsilon \left|\,f(z+v)-f(z)-f'(z)\,v\,\right|\\
-\;d\,\eta_1\, v\big[f(z+v)-f(z)\big]\;,
\end{multline*}
where $z=u+\mu$. Using the star-shapedeness of the domain $\Omega$, we have
\[
\frac{d-2}2\int_\Omega|\nabla v|^2\,dx\leq\int_\Omega h_\varepsilon(v)\;dx + d\,\eta_1\int_\Omega v\big[f(z+v)-f(z)\big]\;dx\;.
\]
Up to a small change of $\eta_1$, so that Condition \eqref{Cdt:AddLoc} still holds, for $\varepsilon>0$, small enough, we get
\[
\frac 1d\,h_\varepsilon(v)\leq F(z+v)-F(z)-F'(z)\,v-\eta_1\, v\big[f(z+v)-f(z)\big]\,.
\]
As $\varepsilon\to 0_+$, $z$ converges to $\mu$ uniformly and the above right hand side is equivalent to $F(v+\mu)-F(\mu)-F'(\mu)\,v-\eta_1\,v\big[f(v+\mu)-f(\mu)\big]$. For some $\delta>0$, arbitrarily small, we obtain
\[
\frac 1d\,h_\varepsilon(v)\leq (H(\mu,\eta_1)+\delta)\,v^2\,.
\]
{}From \eqref{Eqn:Diff_} multiplied by $v$, after an integration by parts we obtain
\[
\int_\Omega|\nabla v|^2\,dx=\int_\Omega v\big[f(z+v)-f(z)\big]\,dx\;.
\]
Hence we have shown that
\[
\left(\frac{d-2}{2\,d}-\eta_1 \right)\int_\Omega|\nabla v|^2\,dx\leq(H(\mu,\eta_1)+\delta)\int_\Omega|v|^2\,dx\;,
\]
By the Poincar\'e inequality~\eqref{Ineq:Poincare}, the left hand side is bounded from below by
\[
\left(\frac{d-2}{2\,d}-\eta_1 \right)\int_\Omega|\nabla v|^2\,dx\geq\frac 1{C_{\rm P}}\left(\frac{d-2}{2\,d}-\eta_1 \right)\int_\Omega|v|^2\,dx\;.
\]
Summarizing, we have proved that, if $\int_\Omega |v|^2\,dx\neq 0$, then, for an arbitrarily small $\delta>0$, 
\[
\frac 1{C_{\rm P}}\left(\frac{d-2}{2\,d}-\eta_1 \right)\leq H(\mu,\eta_1)+\delta
\]
if $\mu-\bar\mu>0$ is small if $\bar\mu>-\infty$, or $\mu<0$, $|\mu|$ big enough if $\bar\mu=-\infty$. This contradicts \eqref{Cdt:AddLoc} unless $v\equiv 0$.
%---------------------------------------------------------------------
\begin{theorem}\label{Thm:Add} Assume that $\Omega$ is a bounded star-shaped domain in $\R^d$, with $C^{2,\gamma}$ boundary, $\gamma\in (0,1)$, such that \eqref {Ineq:Poincare} holds. If $f\in C^2$ satisfies~\eqref{Sc1} and \eqref{Cdt:AddLoc}, if $\lim_{\mu\to\bar\mu}f(\mu)=0$, then there exists a $\mu_0\in(\bar\mu,\infty)$ such that Equation~\eqref{Eqn:AddLoc} has at most one solution in $L^\infty\cap H^1_0(\Omega)$ for any $\mu\in(\bar\mu,\mu_0)$.
\end{theorem}
%---------------------------------------------------------------------
In cases of practical interest for applications, one often has to deal with the equation $\Delta u+f(x,u+\mu)=0$. Our method can be adapted in many cases, that we omit here for simplicity. The necessary adaptations are left to the reader.

%%%%%%%%%%%%%%%%%%%%%%%%%%%%%%%%%%%%%%%%%%%%%%%%%%%%%%%%%%%%%%%%%%%%%%
\subsection{The non-local bifurcation problem}\label{Sec:Add-NonLocal}

In this section we address problem \eqref{Non-a-Mno} with parameter $M>0$, in a bounded star-shaped domain $\Omega$ in $\R^d$. Consider in $L^\infty\cap H^1_0(\Omega)$ the positive solutions of \eqref{Non-a-Mno}, that is of
\be{Non-a}
\Delta u + f(u+\mu)=0
\ee
where $\mu$ is determined by the non-local normalization condition
\be{Mno}
M=\int_\Omega f(u+\mu)\;dx\;.
\ee
We observe that in the exponential case, $f(u)=e^u$, \eqref{Non-a-Mno} is equivalent to the non-local multiplicative case, \eqref{Eqn:MultNonLoc}. The condition \eqref{Mno} is indeed explicitly solved by $e^\mu\int_\Omega e^u\,dx=M=\coupling$.

\medskip Non-existence results for large values of $M$ can be achieved by the same method as in the multiplicative non-local case. If we multiply \eqref{Non-a} by $u$ and $(x\cdot\nabla u)$, we get
\begin{eqnarray*}
&&\int_\Omega|\nabla u|^2\;dx=\int_\Omega u\,f(u+\mu)\;dx\;,\\
&&\frac{d-2}2\int_\Omega|\nabla u|^2\,dx+\frac 12\int_{\partial\Omega}|\nabla u|^2\,(x\cdot\nu)\,d\sigma=d\,\int_\Omega \big(F(u+\mu)-F(\mu)\big)\;dx\;.
\end{eqnarray*}
The elimination of $\int_\Omega|\nabla u|^2\,dx$ gives
\[
\int_\Omega \big[2\,d\,\big(F(u+\mu)-F(\mu)\big)-(d-2)\,u\,f(u+\mu)\big]\,dx\ge\int_{\partial\Omega}|\nabla u|^2\,(x\cdot\nu)\;d\sigma\;.
\]
By the Cauchy-Schwarz inequality, we know that
\[
M^2=\left(\int_{\partial\Omega}\nabla u\cdot \nu\;d\sigma\right)^2\leq |\partial\Omega|\int_{\partial\Omega}|\nabla u|^2\;d\sigma\;.
\]
If~\eqref{star-shapedness} holds, then, as in Section \ref{Sec:Mult-NonLocal},
\[
\alpha\,M^2\leq |\partial\Omega|\int_{\partial\Omega}|\nabla u|^2\,(x\cdot\nu)\;d\sigma\;.
\]
Summarizing we have found that
\be{Ineq:BoundForFermi}
\int_\Omega \big[2\,d\,\big(F(u+\mu)-F(\mu)\big)-(d-2)\,u\,f(u+\mu)\big]\,dx\ge\frac{\alpha\,M^2}{|\partial \Omega|}\;.
\ee
This suggests a condition similar to the one in the multiplicative case, \eqref{Non-mul}. Define
\[
G(\mu):=\sup_{z>\mu}\big[2\,d\,\big(F(z)-F(\mu)\big)-(d-2)\,f(z)\,(z-\mu)\big]\,/\,f(z)\;.
\]
If $f$ is supercritical in the sense of \eqref{Sc1}, $G$ is well defined, but in some cases, it also makes sense for $d=2$. For simplicity, we shall assume that $G$ is a non-decreasing function of $\mu$. As a consequence, we can state the following theorem, which generalizes known results on exponential and Fermi-Dirac distributions, \emph{cf.\/}~\cite{MR2043943} and \cite{MR2134461,MR2136979}, respectively.%---------------------------------------------------------------------
\begin{theorem}\label{Thm:Non-add} Assume that $\Omega$ is a bounded domain in $\R^d$, $d\ge 2$, with $C^1$ boundary satisfying~\eqref{star-shapedness} for some $\alpha>0$. If $f$ is a $C^1$ positive, non-decreasing function such that \eqref{Sc1} holds and if $G$ is non-decreasing, then \eqref{Non-a-Mno} has no solution in $L^\infty\cap H^1_0(\Omega)$~if
\[
M>\frac{|\partial\Omega|}\alpha\;(G\circ f^{-1})\left({\textstyle \frac M{|\Omega|}}\right)\,.
\]
\end{theorem}
%---------------------------------------------------------------------
\noindent Here by $f^{-1}$ one has to understand the generalized inverse given by $f^{-1}(t):=\sup\left\{s\in\R:f(s)\le t \right\}$.
\begin{proof} From the above definitions and computations, we have
\[
\frac{\alpha\,M^2}{|\partial \Omega|}\leq G(\mu)\,M\;.
\]
Since $f$ is non-decreasing and the solution $u$ of \eqref{Non-a} is positive, while $M=\int_\Omega f(u+\mu)\,dx\geq f(\mu)\,|\Omega|$, this completes the proof. \end{proof}

\noindent Theorem~\ref{Thm:Non-add} can be illustrated by the following examples. \begin{enumerate}
\item\emph{Exponential case:\/} if $f(u)=e^{u}$ and $d\geq 3$, then $G(\mu)\equiv d+2+(d-2)\,\log(\frac{d-2}{2\,d})$ does not depend on $\mu$. If $d=2$, $G(\mu)\equiv 4$. In both cases \eqref{Non-a-Mno} has no bounded solution if $M>|\partial\Omega|\,G/\alpha$. We recover here the condition corresponding to~\eqref{ExpSharp} and Theorem~\ref{Thm:Non-mult}.
\item\emph{Power law case:\/} if $f(u)=u^p$ with $p\geq\frac{d+2}{d-2}$, then $G(\mu)=\mu\,G(1)$. Using $\mu\leq(M/|\Omega|)^{1/p}$, it follows that \eqref{Non-a-Mno} has no bounded solution if
\[
M^\frac{p-1}p>\frac{G(1)}\alpha\,\frac{|\partial\Omega|}{|\Omega|^{1/p}}\;.
\]
\item\emph{Fermi-Dirac distribution case:\/} If $f(u)=f_\delta(u)$ where $f_\delta$ is the Fermi-Dirac distribution defined by \eqref{Fermi} with $\delta=d/2-1$ and $d>2\,(1+\sqrt{2})$, then $f$ is increasing, $F=\tfrac 2d\,f_{d/2}$ is the primitive of $f$ such that $\lim_{u\to-\infty}F(u)=0$,
\[
G_d:=\sup_{z\in\R}\left[4\,f_{d/2}(z)-(d-2)\,z\,f_{d/2-1}(z)\right]=\sup_{z\in\R}\left[2\,d\,F(z)-(d-2)\,z\,f(z)\right]
\]
is finite according to \cite[Appendix]{MR2143357} and depends only on the dimension $d$. It is indeed known that $f_\delta'=\delta\,f_{\delta-1}$, $f_\delta(z)\sim\Gamma(\delta+1)\,e^z$ as $z\to -\infty$ and $f_\delta(z)\sim u^{\delta+1}/(\delta+1)$ as $z\to +\infty$. From \eqref{Ineq:BoundForFermi}, we deduce that
\[
\frac{\alpha\,M^2}{|\partial \Omega|}\le\int_\Omega\big[2\,d\,\big(F(z)-F(\mu)\big)-(d-2)\,z\,f(z)\big]\,dx+(d-2)\int_\Omega \mu\,f(z)\;dx
\]
with $z:=u+\mu$. By dropping the term $F(\mu)$, we see that the first integral in the right hand side is bounded by $G_d\,|\Omega|$, and the second 
one by $(d-2)\mu\,M$. Since $f$ is increasing and $u$ positive, $f(\mu)\,|\Omega|\le\int_\Omega f(z)\,dx=M$ and therefore $\mu\le 
f^{-1}(M/|\Omega|))$. As 
a consequence, \eqref{Non-a-Mno} has no bounded solution if
\[
\alpha\,M^2>|\partial\Omega|\,\Big[\;G_d\,|\Omega|+(d-2)\,M\,f^{-1}\!\left(\tfrac M{|\Omega|}\right)\Big]\,.
\]
For a similar approach, one can refer to \cite{MR2136979}.
\end{enumerate}

\medskip Denote by $u_\mu$ a branch of solutions of \eqref{Eqn:AddLoc} satisfying \eqref{H1_}. For $\mu-\bar\mu>0$, small if $\bar\mu>-\infty$, or $\mu<0$, $|\mu|$ big enough if $\bar\mu=-\infty$, a branch of solutions of~\eqref{Non-a-Mno} can be parametrized~by $\mu\mapsto\left(\textstyle M(\mu):=\int_\Omega f(u_\mu+\mu)\,dx\,,\;u_\mu\right)$. Reciprocally, if $\Omega$ is bounded, then any solution $u\in L^\infty\cap H^1_0(\Omega)$ of~\eqref{Non-a-Mno} is of course a solution of \eqref{Eqn:AddLoc} with $\mu=\mu(M)$ determined by \eqref{Mno}. If $f$ is monotone increasing, we additionally know that $\bar\mu<\mu<f^{-1}(M/|\Omega|)$. To prove the uniqueness in $L^\infty\cap H^1_0(\Omega)$ of the solutions of~\eqref{Non-a-Mno}, it is therefore sufficient to establish the monotonicity of $\mu\mapsto M(\mu)$. Assume that
\be{Monotonicity-b}
\mbox{\em $\displaystyle\lim_{\mu\to\bar\mu}f(\mu)=\lim_{\mu\to\bar\mu}f'(\mu)=0$ and $f$ is monotone increasing on $(\bar\mu,\infty)$ .}
\ee
The function $v:=du_\mu/d\mu$ is a solution in $H^1_0(\Omega)$ of
\[
\Delta v+f'(u_\mu+\mu)\,(1+v)=0\;.
\]
As in the proof of Corollary \ref{Cor:Multi1}, by the Maximum Principle, $v$ is nonnegative when $\mu$ is in a right neighborhood of $\bar\mu$, thus proving that
\[
\frac{dM}{d\mu}=\int_\Omega f'(u_\mu+\mu)\,(1+v)\;dx
\]
is nonnegative. Using Theorem~\ref{Thm:Add}, we obtain the following result.
%---------------------------------------------------------------------
\begin{theorem}\label{Thm:Addi1} Assume that $\Omega$ is a bounded star-shaped domain in $\R^d$ with $C^{2,\gamma}$ boundary. If $f\in C^2$ is nonnegative, increasing, satisfies \eqref{Ineq:Poincare}, \eqref{Sc1}, \eqref{Cdt:AddLoc}, and \eqref{Monotonicity-b}, then there exists $M_0>0$ such that \eqref{Non-a-Mno} has at most one solution in $L^\infty\cap H^1_0(\Omega)$ for any $M\in(0,M_0)$.
\end{theorem}
%---------------------------------------------------------------------

%%%%%%%%%%%%%%%%%%%%%%%%%%%%%%%%%%%%%%%%%%%%%%%%%%%%%%%%%%%%%%%%%%%%%%
%%%%%%%%%%%%%%%%%%%%%%%%%%%%%%%%%%%%%%%%%%%%%%%%%%%%%%%%%%%%%%%%%%%%%%
\section{Concluding remarks}

Uniqueness issues in nonlinear elliptic problems are difficult questions when no symmetry assumption is made on the domain. In this paper we have considered only a few simple cases, which illustrate the efficiency of the approach based on Poho\v{z}aev's method when dealing with bifurcation problems. Our main contribution is to extend what has been done in the local multiplicative case to the additive case, and then to problem with non-local terms or constraints.

The key point is that Poho\v{z}aev's method, which is well known to provide non-existence results in supercritical problems, also gives uniqueness results. One can incidentally notice that non-existence results in many cases, for instance supercritical pure power law, are more precisely non-existence results of non-trivial solutions. The trivial solution is then the unique solution.

The strength of the method is that minimal geometrical assumptions have to be done, and the result holds true even if no symmetry can be expected. As a non-trivial byproduct of our results, when the domain $\Omega$ presents some special symmetry, for instance with respect to an hyperplane, then it follows from the uniqueness result that the solution also has the corresponding symmetry.

%%%%%%%%%%%%%%%%%%%%%%%%%%%%%%%%%%%%%%%%%%%%%%%%%%%%%%%%%%%%%%%%%%%%%%
%%%%%%%%%%%%%%%%%%%%%%%%%%%%%%%%%%%%%%%%%%%%%%%%%%%%%%%%%%%%%%%%%%%%%%
\medskip\begin{spacing}{0.8}\noindent\emph{Acknowledgments.\/} {\footnotesize J.D. thanks J.-P. Puel for explaining him the method in the local, multiplicative case, for the exponential nonlinearity, $f(u)=e^u$, and X. Cabr\'e for pointing to him \cite{MR1368678}. The authors thank a referee for pointing them an important missing assumption and M. Jakszto for pointing them references on the Poincar\'e / Friedrichs inequality. This work has been partially supported by the EU financed network HPRN-CT-2002-00282 and the Polonium contract no. 13886SG. It was also partially supported by the European Commission Marie Curie Host Fellowship for the Transfer of Knowledge ``Harmonic Analysis, Nonlinear Analysis and Probability'' MTKD-CT-2004-013389 and by the Polish Ministry of Science project N201 022 32/0902.}\end{spacing}

\smallskip\noindent{\scriptsize\copyright~2009 by the authors. This paper may be reproduced, in its entirety, for non-commercial purposes.}

%%%%%%%%%%%%%%%%%%%%%%%%%%%%%%%%%%%%%%%%%%%%%%%%%%%%%%%%%%%%%%%%%%%%%%
%%%%%%%%%%%%%%%%%%%%%%%%%%%%%%%%%%%%%%%%%%%%%%%%%%%%%%%%%%%%%%%%%%%%%%
\linespread{0.75}

\nocite*\bibliographystyle{siam}\bibliography{References}\end{document}